\documentclass[12pt]{amsart}
\usepackage{amssymb}
\usepackage{amsmath}
\usepackage[hidelinks]{hyperref}
\usepackage{etoolbox,mathtools}
\usepackage{enumitem}
\usepackage{xcolor}
\usepackage{tikz-cd}
\usepackage[normalem]{ulem}
\hypersetup{
	colorlinks,
	linkcolor={red!30!black},
	citecolor={blue!50!black},
	urlcolor={blue!80!black}
}
\usepackage{yhmath}

\newtheorem{theorem}{Theorem}%[section]
\newtheorem{lemma}[theorem]{Lemma}

\newtheorem{definition}[theorem]{Definition}
\newtheorem{corollary}[theorem]{Corollary}

\theoremstyle{remark}
\newtheorem{remark}[theorem]{Remark}

\newcommand{\p}{\vskip .4cm}

\newcommand{\Z}{\mathbb{Z}}
\newcommand{\Q}{\mathbb{Q}}
\newcommand{\R}{\mathbb{R}}
\newcommand{\C}{\mathbb{C}}
\newcommand{\Cc}{\C^{\times}}

\newcommand{\cW}{\mathcal{W}}

\newcommand{\cO}{\mathcal{O}}
\newcommand{\cN}{\mathcal{N}}

\newcommand{\bG}{\mathbb{G}}
\newcommand{\Gm}{\mathbb{G}_m}

\newcommand{\what}{\widehat}

\newcommand{\ra}{\rightarrow}
\newcommand{\ira}{\hookrightarrow}
\newcommand{\sra}{\twoheadrightarrow}
\newcommand{\xra}{\xrightarrow}

\newcommand{\Lie}{\operatorname{Lie}}

\newcommand{\depth}{\operatorname{depth}}
\newcommand{\Ad}{\operatorname{Ad}}

\newcommand{\Irr}{\operatorname{Irr}}

\newcommand{\fg}{\mathfrak{g}}

\newcommand{\fp}{\mathfrak{p}}

\newcommand{\fm}{\mathfrak{m}}

\newcommand{\Hom}{\operatorname{Hom}}

\newcommand{\cha}{\operatorname{char}}

\newcommand{\GL}{\mathrm{GL}}
\newcommand{\SL}{\mathrm{SL}}

\newcommand{\meas}{\operatorname{meas}}

\newcommand{\nild}{(\fg^*)^{\mathrm{nil}}/\!\!\sim}

\begin{document}

\title{Local characters of mod-$\ell$ representations of a $p$-adic reductive group}

\author{Cheng-Chiang Tsai}
\thanks{The author is supported by Taiwan NSTC grants 114-2115-M-001-009 and 114-2628-M-001-003.}
\address{Institute of Mathematics, Academia Sinica, 6F, Astronomy-Mathematics Building, No. 1,
Sec. 4, Roosevelt Road, Taipei, Taiwan \vskip.2cm
also Department of Applied Mathematics, National Sun Yat-Sen University, and Department of Mathematics, National Taiwan University}

\email{chchtsai@gate.sinica.edu.tw}

\begin{abstract} 
    We define the ``lifted character'' of mod-$\ell$ representations of $p$-adic reductive groups where $\ell\not=p$, on compact elements with pro-orders not divisible by $\ell$. We generalize the local character expansion results of Howe, Harish-Chandra and DeBacker to such lifted characters. We show that the result of M\oe glin--Waldspurger and Varma on degenerate Whittaker models is valid for the character expansion.
\end{abstract}

\makeatletter
\let\@wraptoccontribs\wraptoccontribs
\makeatother

\maketitle

\tableofcontents

\section{The Harish-Chandra--Howe local character expansion}

Let $F$ be a non-archimedean local field with residual characteristic $p$. Let $\bG$ be a connected reductive group over $F$ and $G=\bG(F)$. Fix once and for all a Haar measure on $G$ for which $\meas(H)\in p^{\Z}$ for any open pro-$p$ subgroup $H\subset G$. Consequently $\meas(J)\in\Z[1/p]$ for any open compact subgroup $J\subset G$. Consider a perfect field $C$ of characteristic $\ell\not=p$. We will assume that $C$ is countable (see Remark \ref{rmk:count}), and that there exists $\psi:F\ra C^{\times}$ a non-trivial homomorphism with open kernel. When $\ell>0$, let $W(C)$ be the ring of Witt vectors over $C$ and $K(C)=W(C)[1/\ell]$ its field of fractions. When $\ell=0$ we can put\footnote{In this case, many results here are probably known to some experts but not written down before.} $W(C)=K(C)=C$. Let us consider a few notions that appeared in e.g. \cite{Dat12} in related forms.

\begin{definition} By an {\bf $\ell'$-subgroup} of $G$ we refer to an open compact subgroup that is isomorphic to an inverse limit of finite groups of orders not divisible by $\ell$. 
\end{definition}

A representation (over $C$ or $K(C)$) of a Hausdorff topological group is said to be admissible if every vector has an open stabilizer and every open subgroup has a finite-dimensional fixed subspace. Suppose we have an admissible $C$-representation $\pi$ of $G$. Consider $J\subset G$ an $\ell'$-subgroup. Denote by $\pi|_J$ the restriction of $\pi$ to $J$. Then $\pi|_J$ is a semisimple admissible $C$-representation of $J$ and can be lifted to a semisimple admissible $K(C)$-representation $\pi|_J^{\dag}$. The isomorphism class of the lift is unique, see e.g. \cite[Proposition 43]{Ser77}.

For any Hausdorff topological space $X$ and a ring $R$ (henceforth always commutative with identity) we denote by $C^{\infty}_c(X)_R$ the space of compactly supported locally constant $R$-valued functions on $X$. An {\bf $R$-valued distribution} on $X$ will be an $R$-linear functional on $C^{\infty}_c(X)_R$. When $Y\subset X$ is both open and closed, there is an obvious inclusion $C^{\infty}_c(Y)_R\ira C^{\infty}_c(X)_R$ and thus dually we can restrict distributions from $X$ to $Y$.

The admissible representation $\pi|_J^{\dag}$ has its character being a $K(C)$-valued distribution, i.e. a functional
\[
\Theta_{\pi|_J^{\dag}}:C_c^{\infty}(J)_{K(C)}\ra K(C)
\]
When $J_1,J_2$ are $\ell'$-subgroups, so is $J_{12}:=J_1\cap J_2$, and 
\begin{equation}\label{eq:intersect}
\pi|_{J_{12}}^{\dag}=\pi|_{J_{1}}^{\dag}|_{J_{12}}=\pi|_{J_{2}}^{\dag}|_{J_{12}}\implies\Theta_{\pi|_{J_{12}}^{\dag}}=\Theta_{\pi|_{J_1}^{\dag}}|_{J_{12}}=\Theta_{\pi|_{J_2}^{\dag}}|_{J_{12}}
\end{equation}
Therefore, we can glue $\Theta_{\pi|_{J_i}^{\dag}}$ among any collection of $\ell'$-subgroups $J_i$ and have a well-defined distribution on the union $\bigcup J_i$. Let $G_{\ell'}$ be the union of all $\ell'$-subgroups of $G$. We have defined
\[\Theta^{\dag}_{\pi}:C^{\infty}_c(G_{\ell'})_{K(C)}\ra K(C).\]
In other words, given $f\in C^{\infty}_c(G_{\ell'})_{K(C)}$, we decompose the support of $f$ into a disjoint union of open compact subsets, each of which is contained in some $\ell'$-subgroup $J_i$, so that we can evaluate each part under $\Theta_{\pi|_{J_i}^{\dag}}$ and sum them up. Any two such decompositions agree on their common refined decomposition thanks to \eqref{eq:intersect}, and thus give the same sum.

\begin{remark} It is easy to see that
\[G_{\ell'}=\{g\in G\;|\;\lim_{n\ra+\infty} g^{m\cdot p^n}=\mathrm{id}\text{ for some integer }m\text{ not divisible by }\ell\}.\]
\end{remark}

We can now state our first major observation:
\begin{theorem}\label{thm:adm} Suppose we have $\gamma\in G$, two open compact subgroups $H\subset J$, and $\rho^{\dag}\in\Irr_{K(C)}(H)$ such that
\begin{enumerate}[label=(\roman*)]
    \item The fixed subspace $\pi^J\not=0$.
    \item $\gamma\cdot J$ is contained in an $\ell'$-subgroup $I\subset G$.
    \item The distribution $\Theta^{\dag}_{\pi}$ satisfies
    \begin{equation}\label{eq:adm}
    \Theta^{\dag}_{\pi}|_I*\Theta_{\rho^{\dag}}\not\equiv 0.
    \end{equation}
    \end{enumerate}
Then there exists an element $g\in G$ such that $\rho^{\dag}$ has a non-zero fixed vector under $H\cap {}^gJ$ where ${}^gJ:=gJg^{-1}$.

In particular, the distribution $\Theta^{\dag}_{\pi}$ is {\bf admissible} in the sense of Howe and Harish-Chandra \cite[\S15]{HC99}. 
\end{theorem}

\begin{proof} The distribution $\Theta^{\dag}_{\pi}$ is obviously $G$-invariant. To show the rest, since $\Theta^{\dag}_{\pi}|_I=\Theta_{\pi^{\dag}_I}$, equation \eqref{eq:adm} is equivalent to that $\pi^{\dag}_H=\pi^{\dag}_I|_H$ contains $\rho^{\dag}$. Since lifting from $C$ to $K(C)$ is unique, this is equivalent to that $\pi|_H$ contains $\rho$, the latter being the reduction of $\rho^{\dag}$. 

Since $H$ is also an $\ell'$-subgroup, we have an $H$-equivariant surjective map $s:\pi\sra\rho$. Choose non-zero $v\in\pi$ that is fixed by $J$. Since $\pi$ is irreducible, there exists $g\in G$ such that $s(g.v)\not=0$.  Then $g.v$ is fixed by ${}^gJ$ and thus $s(g.v)$ is fixed by $H\cap {}^gJ$. That is, the restriction of $\rho$ to $H\cap {}^gJ$ contains a trivial representation. Lifting $\rho$ to $\rho^{\dag}$, we know that the restriction of $\rho^{\dag}$ to $H\cap {}^gJ$ also contains a trivial representation.
\end{proof}

%\begin{remark} We remark that the idea somewhat resembles \cite[\S II.5]{Vig96}. Indeed, a stronger version for a special case of the theorem is exactly about degenerate Moy-Prasad type that we will use in the proof Corollary \ref{thm:DB} below.\end{remark}

For any ring $R$ and open compact subset $J\subset G_{\ell'}$ we denote by $[J]_R\in C_c^{\infty}(G_{\ell'})_R$ the function that takes the value $1\in R$ in $J$ and $0$ elsewhere. When $R=K(C)$ we will often abbreviate it to $[J]$. Observe that $C_c^{\infty}(G_{\ell'})_R$ has a free basis given by a collection of $[J]_R$'s. This gives
\begin{lemma}\label{lem:BC} For any ring homomorphism $R\ra S$ (required to send $1$ to $1$) there is a natural isomorphism
\begin{equation}
C_c^{\infty}(G_{\ell'})_S=C_c^{\infty}(G_{\ell'})_R\otimes_RS.
\end{equation}
\end{lemma}

Recall that $\psi:F\ra C^{\times}$ is a non-trivial homomorphism with open kernel. Let $\psi^{\dag}:F\ra W(C)^{\times}\subset K(C)^{\times}$ be the unique (Teichm\"{u}ller) lift. Write $\fg:=\Lie\bG(F)$ and $\fg^*:=\Lie^*(\bG)(F)$. Let $\nild$ be the set of nilpotent $\Ad^*(G)$-orbits in $\fg^*$. We are ready to state our Harish-Handra--Howe local character expansion:

\begin{corollary}\label{cor:LCE} Suppose $\cha(F)=0$. There exists an open subgroup $U_{\pi}\subset G$ and a collection $c_{\cO}(\pi,\psi)\in\Q$ for $\cO\in\nild$ such that for any function  $f\in C_c^{\infty}(U_{\pi})_{K(C)}$ supported on $U_{\pi}$, we have \begin{equation}\label{eq:LCE}
\Theta^{\dag}_{\pi}(f)=\sum_{\cO\in(\fg^*)^{\mathrm{nil}}/\sim}c_{\cO}(\pi,\psi)\cdot\mu_{\cO}(\widehat{f\circ\exp})\in K(C).
\end{equation}
\end{corollary}

See Appendix \ref{app:harm} for an explanation of orbital integrals and Fourier transforms with coefficients in $K(C)$. We will generalize Corollary \ref{cor:LCE} to Corollary \ref{cor:LCE2} and give the proof altogether. Let us also mention that \cite{Tsa25b} contains a result similar to Corollary \ref{cor:LCE}.

We continue to assume $\cha(F)=0$. Let $\gamma\in G_{\ell'}$ be semisimple. We denote by $Z_{\fg}(\gamma)^*/\!\!\sim$ the set of nilpotent $\Ad(Z_G(\gamma))$-orbits in $Z_{\fg}(\gamma)^*$. Consider the map $\alpha_{\gamma}:\left(G/Z_G(\gamma)\right)\times Z_{\fg}(\gamma)\ra G$ by $(g,X)\mapsto g(\gamma\cdot\exp(X))g^{-1}$. It is easy to see that $\alpha_{\gamma}$ is a diffeomorphism at the identity. Hence there exist open compact subsets $U_{\gamma}^0\subset G/Z_G(\gamma)$, $U_{\gamma}^1\subset Z_{\fg}(\gamma)$ and $U_{\gamma}\subset G$ such that
\begin{enumerate}
    \item We have $\mathrm{id}\cdot Z_G(\gamma)\subset U_{\gamma}^0$, $0\in U_{\gamma}^1$ and $\gamma\in U_{\gamma}$.
    \item The map $\alpha_{\gamma}$ sends $U_{\gamma}^0\times U_{\gamma}^1$ isomorphically to $U_{\gamma}$.
\end{enumerate}
Fix such $U_{\gamma}^0$, $U_{\gamma}^1$ and $U_{\gamma}$. Choose measures on $Z_G(\gamma)$ and $Z_{\fg}(\gamma)$ so that again every open pro-$p$ subgroup has measure in $p^{\Z}$, and the two measures agree under the exponential map. Give $G/Z_G(\gamma)$ the quotient measure. Then every $f\in C_c^{\infty}(U_{\gamma})_{K(C)}$ can be integrated to $\int_{U_{\gamma}^0}f\circ \alpha_{\gamma}\in C_c^{\infty}(U_{\gamma}^1)_{K(C)}$. %We can write the following generalization of Corollary \ref{cor:LCE}:

\begin{corollary}\label{cor:LCE2} Suppose $\cha(F)=0$ and let $\gamma\in G_{\ell'}$ be semisimple as above. There exists a lattice $U_{\gamma,\pi}^1\subset U_{\gamma}^1$ and a collection $c_{\gamma,\cO}(\pi,\psi)\in K(C)$ for $\cO\in Z_{\fg}(\gamma)^*/\!\!\sim$ such that for any function $f\in C_c^{\infty}(\alpha_{\gamma}(U_{\gamma}^0\times U_{\gamma,\pi}^1))_{K(C)}$, we have
\begin{equation}\label{eq:LCE2}
\Theta^{\dag}_{\pi}(f)=\sum_{\cO\in Z_{\fg}(\gamma)^*/\sim}c_{\gamma,\cO}(\pi,\psi)\cdot\mu_{\cO}(\widehat{\int_{U_{\gamma}^0}f\circ \alpha_{\gamma}})\in K(C)
\end{equation}
\end{corollary}

\begin{proof}  Since $C$ is countable, we may fix a field embedding $\iota:K(C)\ira\C$. Lemma \ref{lem:BC} guarantees that we can base change $\Theta^{\dag}_{\pi}$ along $\iota$. With the admissibility in Theorem \ref{thm:adm}, result of Harish-Chandra \cite[Theorem 16.2]{HC99} gives the existence of neighborhood $U_{\gamma,\pi}^1$ and $c_{\gamma,\cO}(\pi,\psi_{\C})\in\C$ (that depend on $\psi_{\C}:F\xra{\psi^{\dag}}K(C)\xhookrightarrow{\iota}\C$) such that \eqref{eq:LCE2} holds for $f\in C_c^{\infty}(\alpha_{\gamma}(U_{\gamma}^0\times U_{\gamma,\pi}^1))_{\C}$.

We claim that $c_{\gamma,\cO}(\pi,\psi_{\C})\in\iota(K(C))\subset\C$. Suppose this is true for all larger orbits. Pick a lattice $\Lambda\subset Z_{\fg}(\gamma)^*$ and $n\in\cO$ such that every $\cO'\in Z_{\fg}(\gamma)^*/\!\!\sim$ intersecting $n+\Lambda$ is either larger or equal to $\cO$. Consider $h=\widehat{[n+\Lambda]}\in C_c^{\infty}(Z_{\fg}(\gamma))_{\iota(K(C))}$, where the Fourier transform is defined using $\psi_{\C}$. By scaling $n$ and $\Lambda$ simultaneously by $\varpi_F^{-2n}$ for $n\gg 0$ we may assume $h$ is supported in $U_{\gamma,\pi}^1$. Let $f\in C_c^{\infty}(\alpha_{\gamma}(U_{\gamma}^0\times U_{\gamma,\pi}^1))_{\iota(K(C))}$ be given by $f(\alpha_{\gamma}(u_0,u_1))=h(u_1)$. Thanks to Appendix \ref{app:harm} we have
\[
\mu_{\cO'}(\widehat{\int_{U_{\gamma}^0}f\circ \alpha_{\gamma}})=\operatorname{meas}(U_{\gamma}^0)\cdot\mu_{\cO'}(h)\in\iota(K(C))\;\text{ for any }\;\cO'.
\]
By construction of $h$ we have $\mu_{\cO'}(h)\not=0$ iff $\cO'\ge \cO$. Comparing both sides of \eqref{eq:LCE2} for our chosen $f$ we have $c_{\gamma,\cO}(\pi,\psi_{\C})\in\iota(K(C))$. We can now assign $c_{\gamma,\cO}(\pi,\psi):=\iota^{-1}\left(c_{\gamma,\cO}(\pi,\psi_{\C})\right)$, with which \eqref{eq:LCE2} hold for all $f\in C_c^{\infty}(\alpha_{\gamma}(U_{\gamma}^0\times U_{\gamma,\pi}^1))_{K(C)}$.
\end{proof}

\begin{remark} When $\gamma=\mathrm{id}$ we furthermore have $c_{\cO}(\pi,\psi)=c_{\mathrm{id},\cO}(\pi,\psi)\in\Q$ thanks to \cite[\S7]{Var14}.
\end{remark}

\begin{remark} Let us remark that Corollary \ref{cor:LCE2} gives a different proof to some earlier results of Vign\'{e}ras and of Dat \cite[Th\'eor\`eme 2 (Appendice) and Th\'eor\`eme 2.1.4]{Dat12} about local constancy of mod-$\ell$ characters.

Our distribution $\Theta_{\pi}^{\dag}$ is locally glued from characters of representations of $\ell'$-subgroups over $W(C)$. Hence 
\[f\in C_c^{\infty}(G_{\ell'})_{W(C)}\implies\Theta_{\pi}^{\dag}(f)\in W(C).\]
Fix $\gamma\in G_{\ell'}$ that is regular semisimple. Then Corollary \ref{cor:LCE2} applied with $\iota:K(C)\ra\C$ together with \cite[Theorem 16.1]{HC99} show that $\Theta_{\pi}^{\dag}$ is locally constant at $\gamma$ with some value $T(\gamma)\in K(C)$ in the following sense: There exists an $\ell'$-subgroup $H_{\gamma}$ such that for any open subgroup $H\subset H_{\gamma}$ and any coset $hH\subset H_{\gamma}$ we have
\begin{equation}\label{eq:const1}
\Theta_{\pi}^{\dag}([\gamma h H])=T(\gamma)\cdot\meas(H)\in W(C).
\end{equation}
Since we can choose $H$ with $\meas(H)\in p^{\Z}$, this implies $T(\gamma)\in W(C)$. We can thus take reduction modulo $\ell$, giving
\begin{equation}\label{eq:const2}
\Theta_{\pi}([\gamma hH]_C)=\overline{T(\gamma)}\cdot\overline{\meas(H)}\in W(C).
\end{equation}
where $\Theta_{\pi}$ is the usual character and $\overline{\meas(H)}$ is the image of $\meas(H)\in\Z[1/p]$ in $C$. %It is easy to see that $T(\gamma)$ and thus its reduction $\overline{T(\gamma)}$ does not depend on the choice of $H_{\gamma}$. 
In other words, $\Theta_{\pi}$ is locally constant near $\gamma$ with value $\overline{T(\gamma)}$.
\end{remark}

\section{Waldspurger-DeBacker homogeneity}

We drop the assumption $\cha(F)=0$ and discuss how the Harish-Chandra--Howe local character expansion can be enhanced by the result of Waldspurger \cite{Wa95} and DeBacker \cite{De02a}:

\begin{theorem}\label{thm:DB} Suppose the hypotheses in \cite[\S2.2, \S3.2]{De02a} hold for $\bG$ and $F$, possibly with $\cha(F)>0$. Then Corollary \ref{cor:LCE} holds with $U_\pi$ being the union of all Moy-Prasad subgroups $G_{x,\depth(\pi)+}$, and $\exp$ possibly replaced by a mock exponential map in \cite[\S3.2]{De02a}.
\end{theorem}

\begin{proof} We plug $\Theta_{\pi}^{\dag}$ into the proof of \cite[Theorem 3.5.2]{De02a}. The only property of $\pi$ is used in \cite[Lemma 3.3.2]{De02a} and is the following: When $J=G_{x,s}$ is the Moy-Prasad subgroup, and when $\pi^{\dag}_J$ contains the character given by
\[
\def\arraystretch{1.2}
\begin{array}{ccccccccl}
\psi^{\dag}_{\chi}&:&J=G_{x,s}&\sra& G_{x,s}/G_{x,s+}&\cong&\Hom_k(\fg^*_{x,-s}/\fg^*_{x,(-s)+},\,k)&\ra&K(C)^{\times}\\
&&&&&&\phi&\mapsto &\psi^{\dag}(\phi(\chi))
\end{array}
\]
then $\chi$ is required to be {\bf degenerate}, i.e. $\chi+\fg^*_{x,(-s)+}$ intersects the nilpotent cone non-trivially.  Observe that $\pi|_J^{\dag}$ contains $\psi_{\chi}^{\dag}$ if and only if $\pi|_J$ contains $\psi_{\chi}$. The latter implies that $\chi$ is indeed degenerate by \cite[Th\'{e}or\'{e}me II.5.7.a(a)]{Vig96}.
\end{proof}

\section{Dimensions of fixed subspaces}

\begin{corollary} Suppose $F$ and $\bG$ satisfy either the condition in Corollary \ref{cor:LCE} or the hypotheses used in Theorem \ref{thm:DB}. Let $q$ be the order of the residue field of $F$. Let $P\subset G$ be any parahoric subgroup and $P_n$ be the $n$-th congruence subgroup. Then there exists a polynomial $h_{P}(X)\in\Q[X]$ such that $\dim_C\pi^{P_{2n}}=h_{P}(q^{2n})$ for all integers $n>N_{\pi}$ for some $N_{\pi}$. Moreover $\deg(h_P)=\max\{\frac{1}{2}\dim\cO\;|\;\cO\in\nild,\;c_{\cO}(\pi)\not=0\}$. \end{corollary}

This is well-known when $\cha(C)=0$. When $\mathrm{char}(C)=\ell$ it was shown by Henniart and Vign\'{e}ras for $G=\GL_n(D)$ \cite[Theorem 1.4]{HV24} and for $G=\SL_2(F)$ \cite[Theorem 1.9]{HV25} (with mild modification) without restriction on $p$. We learned it as a conjecture in the general case from Marie-France Vign\'{e}ras.

\begin{proof} Let $c(\psi)$ be the maximal fractional ideal in $F$ such that $\ker(\psi)=\ker(\psi^{\dag})$ contains $c(\psi)$. Let $\fp$ be Lie algebra of $P$ and
\[\fp^{\perp}:=\{\chi\in\fg^*\;|\;\chi(X)\in c(\psi)\text{ for all }X\in\fp\}.\]
Then for any $n\in\Z$, the Fourier transform $\what{[\fm_F^{-n}\cdot\fp^{\perp}]}$ is the function supported on $\fm^n\cdot\fp$ with value equal to the measure of $\fm^n\cdot\fp$ (no matter how the latter is normalized), and thus for $n$ sufficiently large we have
\[
\dim_{C}\pi^{P_{2n}}=\dim_{K(C)}(\pi^{\dagger})^{P_{2n}}=\Theta^{\dagger}_{\pi}([\fm_F^{-2n}\cdot\fp^{\perp}]\circ\exp^{-1}).
\]
The corollary then follows from the fact that $\mu_{\cO}([\fm_F^{-2n}\cdot\fp^{\perp}])=q^{n\dim\cO}\cdot\mu_{\cO}([\fp^{\perp}])$ for any $n\in\Z$.
\end{proof}

\section{Degenerate Whittaker models}

In \cite{MW87}, M\oe glin and Waldspurger defined the so-called degenerate Whittaker model $\cW_{Y,\varphi,\psi_{\C}}(\pi)$ of a $\C$-representation $\pi$ that is defined in terms of\footnote{In both \cite{MW87} and \cite{Var14} an identification $\fg^*\approx\fg$ was chosen. It seems to us that it is not necessary, with some care to work in the dual Lie algebra.} an element $Y\in(\fg^*)^{\mathrm{nil}}$, a cocharacter $\varphi:\Gm\ra\bG$ that acts on $Y$ by weight $-2$, and $\psi_{\C}:F\ra\Cc$. Their definition works directly with $\psi_{\C}$ replaced by $\psi$ and $\pi$ replaced by a $C$-representation. We can therefore generalize their theorem (and also Varma's theorem \cite{Var14} when $p=2$):

\begin{corollary}\label{cor:MW} Suppose $\cha(F)=0$. Let $\psi$ be as above and $\pi$ be an irreducible admissible $C$-representation. Let $Y$ and $\varphi$ be as above such that $\cW_{Y,\varphi,\psi}(\pi)\not=0$. Then
\begin{enumerate}[label=(\roman*)]
    \item There exists $\cO\ge\Ad(G)Y$ that is maximal in $\{\cO'\in\nild\;|\;c_{\cO'}(\pi,\psi)\not=0\}$ where we consider the partial order given by closure relation.
    \item Suppose in (i) we have $\cO=\Ad(G)Y$. Then $c_{\cO}(\pi,\psi)=\dim\cW_{Y,\varphi,\psi}(\pi)\in\Z$.
\end{enumerate}
\end{corollary}

The relation between degenerate Whittaker models and the local character for a mod-$\ell$ representation also was established by Henniart and Vign\'{e}ras for $G=\GL_n(D)$ \cite[\S8]{HV24} and for $G=\SL_2(F)$ \cite[Theorem 6.2]{HV25}.

\begin{proof} We mimic the proof in \cite{Var14}, and therefore indirectly follow \cite{MW87}. However, we replace the appearances of $\C$-representation by $C$-representation (except for occasionally by $K(C)$-representation, which we indicate below), replace the appearances of  $\C$-valued functions by $K(C)$-valued functions, and replace the usual character by $\Theta_{\pi}^{\dag}$. This is the whole idea of ours, but we'll try to be more precise below.

The role of the additive character $\Lambda$ in \cite{Var14} will be played both by our $\psi:F\ra C^{\times}$ and its lift $\psi^{\dag}:F\ra K(C)^{\times}$. Let $\varpi_F\in F$ be a fixed uniformizer. In \cite[\S3]{Var14} a lattice $L\subset\fg$ is carefully chosen, and $G_n:=\exp(\varpi_F^n\cdot L)$ for $n$ large. We define $\chi_n:G_n\ra C^{\times}$ as in \cite[\S4]{Var14}: One way to do it is to first define $\chi_n^{\dag}:G_n\ra K(C)^{\times}$ by pulling back back the $\Cc$-valued character in \cite[\S4]{Var14} along $\iota:K(C)\ira\C$, where $\iota\circ\psi^{\dag}:F\ra\Cc$ will serve as the additive character $\Lambda$ in \cite{Var14}. Since $G_n$ is pro-$p$, we have $\chi_n^{\dag}(G_n)\in W(C)^{\times}$ and we can define $\chi_n$ to be its reduction mod $\ell$. Alternatively, we can re-run the proof in \cite[\S4]{Var14} for mod-$\ell$ representation.

Mimicking the notation in \cite{Var14}, we denote by $W$ the underlying vector space over $C$ of $\pi$. We put
\[
\begin{array}{c}
W_n:=\{w\in W\;|\;\forall \gamma\in G_n,\;\pi(\gamma)w=\chi_n(\gamma)w\}.\\
W_n':=\{w\in W\;|\;\forall \gamma\in G_n',\;\pi(\gamma)w=\chi_n'(\gamma)w\}.
\end{array}
\]
where $G_n':=\mathrm{Int}(\varphi(\varpi^{-n}))G_n$ as in \cite[pp. 1038]{Var14} and $\chi_n'$ is likewise the conjugate of $\chi_n$. As always we choose a Haar measure on $G$ for which every pro-$p$ open compact subgroup has measure in $p^{\Z}$. These measures have well-defined invertible image in $C$, so that as {\it ibid.} we have the maps $I_{n,m}':W_n'\ra W_m'$ that is well-defined by integrating locally constant $C$-valued objects over the pro-$p$ open compact subgroup $G_m'$. All other objects until \cite[(19)]{Var14} can be defined similarly. Then, we consider the local character expansion for $(\pi,W)$ in Corollary \ref{cor:LCE}, and define $\cN_{tr}(\pi)$ similarly to be the set of maximal $\cO\in\nild$ for which $c_{\cO}(\pi,\psi)\not=0\in K(C)$. In the rest of \cite[\S5]{Var14}, we can likewise state (proofs await) analogues of \cite[Proposition 1, Lemmas 7 and 8]{Var14}; note that they are about the mod-$\ell$ representation $(\pi,W)$ and its characteristic $0$ ``local character expansion.'' We can also state \cite[Lemma 9]{Var14}, which involves some measures of pro-$p$ groups and their inverses that should be interpreted as their images in $C$. Then Corollary \ref{cor:MW} is the analogue of \cite[Theorem 1]{Var14} and follows from the analogues of \cite[Proposition 1, Lemmas 7, 8 and 9]{Var14}.

Next, we prove these analogues as in \cite[\S6]{Var14}. We similarly have \cite[Lemma 10]{Var14}, using that $V_n'$ {\it op. cit.} is a pro-$p$ group so that its measure has invertible image in $C$. The interaction between mod-$\ell$ and characteristic $0$ first happens in the Proof of the analogue of \cite[Proposition 1]{Var14}. In the proof Varma made use of \cite[(20) and (21)]{Var14}, which we reproduce (as \eqref{eq:Varma1} and \eqref{eq:Varma2}) as follows: let $\varphi_n^{\dag}\in C_c^{\infty}(G_n)_{K(C)}$ be given by $\varphi_n^{\dag}(\gamma)=\chi_n^{\dag}(\gamma)^{-1}$ for $\gamma\in G_n$. Then
\begin{equation}\label{eq:Varma1}
\dim_C W_n = (\operatorname{meas} G_n)^{-1}\Theta_{\pi}^{\dag}(\varphi_{n}^{\dag})
\end{equation}
\begin{equation}\label{eq:Varma2}
\Theta_{\pi}^{\dag}(\varphi_{n}^{\dag})=(\operatorname{meas}(\varpi^nL)^{\perp})^{-1}=\sum_{\substack{\cO\in(\fg^*)^{\mathrm{nil}}/\sim\\\cO\ge\Ad(G)Y}}c_{\cO}(\pi,\psi)\mu_{\cO}(\widehat{-\varpi^{-2n}Y+(\varpi^nL)^{\perp}}).
\end{equation}
To prove equation \eqref{eq:Varma1}, we have $\dim_C W_n =\dim_{K(C)}W_n^{\dag}$ where $W_n^{\dag}$ is defined to be a lift of $W_n$ as $G_n$-representation. Since $\Theta_{\pi}^{\dag}=\Theta_{\pi|_{G_n}^{\dag}}$, we have \eqref{eq:Varma1} just like \cite[(20)]{Var14} as a property of $\chi_n^{\dag}$. Equation \eqref{eq:Varma2} follows also just like \cite[(21)]{Var14}, except that we use Corollary \ref{cor:LCE} instead of the usual local character expansion. With these equations, the rest of the proof of the analogues of \cite[Lemmas 7, 8 and 9]{Var14} follow similarly.
\end{proof}

We end with the remark that stronger results about $\C$-coefficient characters have been developed based on \cite{De02a,De02b} in a series of works including \cite{KM03,KM06,AK07,AS09,DS18,Spi18,Spi21} and many others. In particular \cite{Spi21} achieved (under hypotheses on $p$) the description of the full character of a regular supercuspidal representation in \cite{Kal19} and consequently their endoscopic character identities are proved in \cite{FKS23}. It seems likely to us that on $G_{\ell'}$ our method will combine just well with all these results. Such a statement is however vast and have several slightly different meanings in different contexts above. We leave them to future investigations.

\subsection*{Acknowledgment} I am truly grateful to Marie-France Vign\'{e}ras for a beautiful talk, and for many very helpful discussions after that, all of which had inspired this work. I sincerely thank Rahul Dalal and Mathilde Gerbelli-Gauthier for their generous and stimulating explanation about their related upcoming paper \cite{DGGM}\footnote{I think there is no overlap of content, but it doesn't hurt to emphasize that the cited work was essentially done and explained to me before this work of mine started.}. I also thank the organizers of the wonderful conferences in the Field Institute in August 2025 and the Paris center of U. Chicago in September 2025. Lastly, I owe an immense debt of gratitude to Stephen DeBacker for teaching me about many things related to his works and for numerous other enlightening discussions that very much influence this work. 

\appendix
\section{Miscellaneous harmonic analysis}\label{app:harm}

Firstly we remark about countability of $C$.

\begin{remark}\label{rmk:count} For the sake of representation theory, we have a trick to go from arbitrary $C$ to the countable case: We are interested in an irreducible admissible $C$-representation $\pi$. Since $G/J$ is countable for any open subgroup $J\subset G$, we know $\dim_C\pi$ is countable. Choose any $C$-basis of $\pi$ and consider the countable collection of vectors $g.v$ for any $v$ in the basis and $g\in G$, written as linear combinations of the basis vectors. Then we can replace $C$ by the algebraic closure in $C$ of the subfield generated by all the coefficients. One can verify that irreducibility, etc. are unaffected by descending to this subfield.
\end{remark}

\subsection{Nilpotent orbital integrals}\label{subsec:orb} In this section, we assume either that $\cha(F)=0$, or $\bG$ is $F$-standard \cite[Definition 3]{Mc04}; this holds for example when $\cha(F)$ is very good for $\bG$. Let $e\in\fg^*$ be nilpotent and $\cO=\Ad^*(G)e$. Originally, the orbital integral
\[
\mu_{\cO}:C_c^{\infty}(\fg^*)_{\C}\ra\C
\]
is well-defined thanks to Ranga Rao \cite[Theorem 1]{Ran72} and also McNinch \cite[Theorem 45]{Mc04}. The definition involves an absolutely converging series in $\C$ (or $\R$). Once defined, the orbital integral lies in the unique $1$-dimensional space of $G$-invariant $\C$-distributions on $\fg^*$ supported on the closure of $\cO$ such that for any lattice $\Lambda\subset\fg^*$, we have 
\begin{equation}\label{eq:easyscale}
\mu_{\cO}([\fm_F^{-2}\Lambda]_{\C})=q^{\dim\cO}\cdot\mu_{\cO}([\Lambda]_{\C})
\end{equation}
We further normalize it with the measure normalized in \cite[Remark 2]{Var14}, which requires $\mu_{\cO}([J]_{\C})$ to be a specific value in $p^{\Z}$ for a specific open compact subset $J$. Choose any $\Q$-linear map $p:\C\sra\Q$ such that $p(a)=a$ for every $a\in\Q$. Then
\[
C_c^{\infty}(\fg^*)_{\Q}\ni f\mapsto p(\mu_{\cO}(f))\in\Q
\]
is a $G$-invariant $\Q$-valued distribution supported on the closure of $\cO$, satisfying \eqref{eq:easyscale}, and has the same normalization. Write $\mu_{\cO,\Q}:=p(\mu_{\cO})$. By Lemma \ref{lem:BC} we can base change $\mu_{\cO,\Q}$ from $\Q$ to $\C$. Such a base change also satisfy the above characterizing property, i.e. it has to be equal to $\mu_{\cO}$. In particular for any $f\in C_c^{\infty}(\fg^*)_{\Q}$ we have $\mu_{\cO}(f)=\mu_{\cO,\Q}(f)=p(\mu_{\cO}(f))\in\Q$. This gives a different proof of the main result of \cite{Ass96}. 

For any characteristic zero field $K$, we base change $\mu_{\cO,\Q}$ from $\Q$ to $K$ using Lemma \ref{lem:BC}. In the main body of this article we denote such base change again by $\mu_{\cO}$ when there is no danger of confusion.

\subsection{Fourier transform} Let $V$ be any finite-dimensional vector space over $F$ and $V^*=\Hom_F(V,F)$ its dual, both equipped with the topology from $F$. Let $R$ be an integral domain (commutative with identity) in which $p$ is invertible, and $\psi_R:F\ra R^{\times}$ be a non-trivial homomorphism with an open kernel. In this subsection we define Fourier transform with respect to $\psi_R$:
\[
\begin{array}{ccc}
C_c^{\infty}(V)_R&\ra&C_c^{\infty}(V^*)_R\\
f&\mapsto&\hat{f}
\end{array}
\]
In the main body of the paper $\psi_R$ will usually be the obvious choice: most of time we have $R=K(C)$, or sometimes $R=\C$, for which we take $\psi_R=\psi^{\dag}$ and $\psi_R=\psi_{\C}$, respectively.

Let $c(\psi_R)\subset F$ be the largest fractional ideal contained in $\ker(\psi_R)$. For any lattice $\Lambda\subset F$ we put
\[
\Lambda^{\perp}:=\{X\in V^*\;|\;\langle X,Y\rangle\in c(\psi_R),\;\forall Y\in\Lambda\}.
\]
Suppose we have measures on $V$ and $V^*$ for which every lattice has measure in $p^{\Z}$ and such that $\meas(\Lambda)\cdot\meas(\Lambda^{\perp})=1$ for every lattice $\Lambda$. Any $f\in C_c^{\infty}(V)_R$ is supported on some lattice $\Lambda_1$ and constant under translation by some lattice $\Lambda_0$. We can view $f$ as a function on the finite abelian $p$-group $\Lambda_1/\Lambda_0$. Then $\hat{f}\in C_c^{\infty}(V^*)_R$ will be defined as a function supported on $\Lambda_0^{\perp}$ and constant under translation by $\Lambda_1^{\perp}$. By viewing $\hat{f}$ as a function on $\Lambda_0^{\perp}/\Lambda_1^{\perp}\cong\Hom_{\cO_F}(\Lambda_1/\Lambda_0,F/c(\psi_R))$, it is defined as
\[
\hat{f}(Y)=\frac{1}{\meas(\Lambda_0)}\sum_{X\in\Lambda_1/\Lambda_0}\psi_R(\langle X, Y\rangle)\cdot f(X).
\]
Here $\meas(\Lambda_0)^{-1}$ should be understood as its image under $\Z[1/p]\ra R$. We can likewise go from $V^*$ to $V$: It is easy to see that $(\Lambda^{\perp})^{\perp}=\Lambda$, and $\hat{\hat{f}}(X)=f(-X)$. Moreover, taking Fourier transform commutes with base change for $R$.

\p
\bibliographystyle{amsalpha}
\def\cfgrv#1{\ifmmode\setbox7\hbox{$\accent"5E#1$}\else \setbox7\hbox{\accent"5E#1}\penalty 10000\relax\fi\raise 1\ht7 \hbox{\lower1.05ex\hbox to 1\wd7{\hss\accent"12\hss}}\penalty 10000 \hskip-1\wd7\penalty 10000\box7}
\providecommand{\bysame}{\leavevmode\hbox to3em{\hrulefill}\thinspace}
\providecommand{\MR}{\relax\ifhmode\unskip\space\fi MR }
% \MRhref is called by the amsart/book/proc definition of \MR.
\providecommand{\MRhref}[2]{%
  \href{http://www.ams.org/mathscinet-getitem?mr=#1}{#2}
}
\providecommand{\href}[2]{#2}

\end{document}